\documentclass[preprint,12pt]{elsarticle}

\usepackage{amssymb}
\usepackage{amsthm}
\usepackage{amsmath}
\newtheorem{theorem}{Theorem}[section]
\newtheorem{cor}[theorem]{Corollary}
\newtheorem{lemma}[theorem]{Lemma}

\newtheorem{rem}[theorem]{Remark}
\newcommand{\Fq}{\mathbb{F}_q}

\begin{document}

\begin{frontmatter}
\title{Large sets of Kirkman triple systems with order $q^n+2$}
\ead{wpolly0419@gmail.com}

\author[lzu]{C. Wang\corref{cor1}}
\author[lzu]{C. Shi}
\address[lzu]{School of Mathematics and statistics, Lanzhou University, Gansu 730000, P.R. China}
\cortext[cor1]{Corresponding Author}

\begin{abstract}
The existence of Large sets of Kirkman Triple Systems (LKTS) is an old problem in combinatorics. Known results are very limited, and a lot of them are based on the works of Denniston \cite{MR0349416, MR0369086, MR535159, MR539718}. The only known recursive constructions are an tripling construction by Denniston \cite{MR535159}and a product construction by Lei \cite{MR1931492}, both constructs an LKTS($uv$) on the basis of an LKTS($v$).

In this paper, we describe an construction of LKTS$(q^n+2)$ from LKTS$(q+2)$, where $q$ is a prime power of the form $6t+1$. We could construct previous unknown LKTS($v$) by this result, the smallest among them have $v=171,345,363$.
\end{abstract}
\end{frontmatter}

\section{Introduction}
A \emph{Steiner triple system} of order $v$, denoted by STS($v$), is a pair $(X, \mathcal{B})$, where $X$ is a $v$-set and $\mathcal{B}$ is a collection of triples (called \emph{blocks}) of $X$ such that every pair of $X$ is contained in a unique block of $B$. Let $(X, \mathcal{B})$ be an STS($v$). If there exists a partition $A = \{P_1, P_2,\ldots,P_{(v-1)/2}\}$ of $\mathcal{B}$ such that each part $P_i$ forms a parallel class, i.e., a partition of $X$, then the STS($v$) is called \emph{resolvable} and $A$ is called a \emph{resolution}. A resolvable STS($v$) is usually called a Kirkman triple system of order $v$ (briefly KTS($v$)). It is well known that a KTS($v$) exists if and only if $v\equiv3\pmod{6}$.

Two STS($v$)s on the same set $X$ are said to be \emph{disjoint} if they have no triples in common. By a simple counting argument, there can be at most $v-2$ mutually disjoint STS($v$)s on the same set $X$;
such a set of $v-2$ disjoint systems must contain every possible triple in $X$, and is called a \emph{large set} of STS($v$)s and denoted by LSTS($v$). A large set of Kirkman triple systems of order $v$, denoted by LKTS($v$), is an LSTS($v$) where each STS($v$) is resolvable.

The existence problem of LKTS is posed by Sylvester in 1850 as an extension of Kirkman's schoolgirl problem. The existence of LSTS has been completely solved by Lu\cite{MR692824,MR692825,MR692826,MR757612} and Teirlinck\cite{MR1111565}, but existence of LKTS is much harder to establish, and only limited results are known. Kirkman \cite{LKTS9} showed the existence of an LKTS(9) in 1850, and Denniston \cite{MR0369086} constructed an LKTS(15) more than one century later. Denniston\cite{MR0349416,MR539718}, Chang and Ge\cite{MR1675197,MR2054972} and Zhou and Chang \cite{MR2593325} also gave further direct construction for LKTS. The first recursive construction of LKTS was a tripling construction by Denniston \cite{MR535159}, which was later generalized by Lei\cite{MR1931492} to a product construction. Zhang and Zhu \cite{MR1893029,MR1948399} improved the above mentioned tripling and product constructions, removing the need of transitive KTS. The known result for existence of LKTS are summarized below:
\begin{theorem}
\begin{enumerate}
\item There exists an LKTS($3^a5^br\prod^s_{i=1}(2\cdot13^{n_i}+1)\prod^t_{j=1}(2\cdot7^{m_j}+1)$) for integers $r\in\{1,7,11,13,17,35,43,67,91,123\}\cup\{2^{2p+1}5^{2q}+1|p,q\geq1\},a,n_i,m_j\geq 1,b,s,t\geq 0$ and $a+s+t\geq 2$ when $b\geq 1$ and $r\neq 1$.
\item There exists an LKTS($3\prod^s_{i=1}(2{q_i}^{n_i}+1)\prod^t_{j=1}(4^{m_j}-1)$), where $s+t\geq 1,n_i,m_j\geq 1, q_i\equiv7\pmod{12}$ and $q_i$ is a prime power.
\end{enumerate}
\end{theorem}

To this day, known recursive constructions of LKTS are all product constructions; that is, they construct an LKTS($pq$) from an LKTS($p$) and some auxiliary design on $q$ points. In this paper, we provide a construction of LKTS($q^n+2$) from known LKTS($q+2$), which serves as a good complement to the known product constructions.

\section{Main results}
The main result we will prove in this paper is
\begin{theorem}\label{thm:main}
If there exists an LKTS$(q+2)$ for $q\equiv1\pmod{6}$ a prime power, then there exists an LKTS$(q^n+2)$ for every integer $n\geq 2$.
\end{theorem}

LKTS($q+2$) is known for $q=7, 13, 19, 25, 31, 37, 43, 49, 61$ and many more values. By Theorem \ref{thm:main}, we can deduce that:
\begin{cor}
There exists an LKTS($q^n+2$) for $n\geq 1$ and \\ $q\in\{7, 13, 19, 25, 31, 37, 43, 49, 61\}$.
\end{cor}
Taken into account the product constructions in \cite{MR1948399,MR2438170}, we have the following:
\begin{cor}
There exists an LKTS($3^a5^b(q^n+2)\prod^s_{i=1}(2\cdot13^{n_i}+1)\prod^t_{j=1}(2\cdot7^{m_j}+1)$) for integers $q\in\{7, 13, 19, 25, 31, 37, 43, 49, 61\},n_i,m_j\geq 1,a,b,n,s,t\geq 0$.
\end{cor}

This result enables us to construct an infinite family of previous unknown LKTS, the smallest of which have order $v=171,345,363,513,627,855,963,1035$ and $1089$.

We will construct an LKTS on the point set $W\bigcup\{\infty_1,\infty_2\}$, where $W$ is a vector space of dimension $n$ over $\Fq$, and $\infty_i$ are two additional points. The details of the construction are exhibited in the next section.

\section{The Construction}
Let $q=6t+1$ be a prime power, $(Y=\Fq\bigcup\{\infty_1,\infty_2\},\left.\mathcal{D}_i\right|_{i\in\Fq})$ be a large set of KTS($q+2$), where we take $\mathcal{D}_i$ to be the design that contains the triple $\{\infty_1,\infty_2,i\}$, and $W=\Fq^n$ be a linear space over $\Fq$. Without loss of generality, let $Q_{i,j} (j=0,1,\ldots,(q-1)/2)$ be the parallel classes of $\mathcal{D}_i$, where $\{\infty_1,\infty_2,i\}\in Q_{i,0}$. We will construct an LKTS($q^n+2$) on the point set $X=W\bigcup\{\infty_1,\infty_2\}$.

For each triple $A\subset Y$ and point $x\in W$, we denote by $Ax$ the set $\{ax|a\in A\}$, where $\infty_i x=\infty_i, i=1,2$.

The main ingredient of our construction is a large set of "frames" over $W$, resolvable into partial parallel classes, and invariant under translations in $W$. The triples in the "base" frame are exactly those of the form $\{x,y,-x-y\}$ where $x$ and $y$ are independent. The "holes" within each frame are 1-dimensional affine subspaces of $W$ (or geometrically, lines in $W$) that passes a given point; We make use of the triples from the original large set $Y$ to fill in the holes.


We will only describe the "base" frame, which contains all the zero-sum triples that are non-collinear. It is easy to show that the total number of such triples are $(q^n-q)(q^n-1)^2/6$. We will describe a partition of these triples into $(q^n-1)/2$ partial parallel classes, where each class contains $(q^n-q)/3$ triples.

We fix a family of skew-symmetric bilinear forms $f_L:L\times L \rightarrow\Fq$ for each 2-dimensional subspace $L$ of $W$.

Let $g$ be a primitive root of $q$, and $\omega=g^{2t}$ be a third root of unity in $\Fq$. We define, for each ordered linearly independent pair of points $(u,v)\in W^2$, the triple
\begin{equation}\label{eqn:T}
T(u,v)=\{u+v,\omega u+\omega^2v,\omega^2u+\omega v\}.\end{equation}
It is easy to see that $T(u,v)=T(\omega u,\omega^2 v)=T(\omega^2 u,\omega v)$. The following lemma gives the basis of our construction.

\begin{lemma}\label{lem:partialclass}
Let $K$ be a 1-dimensional subspace of $W$, and $L\supset K$ is an 2-dimensional space containing $K$. For any $u\in K$ and $c\in \Fq^*$, the $q(q-1)/3$ triples $\{\pm T(u,v)|v\in L, f_L(u,v)=g^mc,0\leq m<t\}$ is a partition of $L \setminus K$. We will denote this set by $P_{u,L,c}$.
\end{lemma}

\begin{proof}
As $L\setminus K$ have $q(q-1)$ elements, it suffices to prove that every element in $L \setminus K$ belongs to one of such triple.

Let $v_0\in L$ such that $f_L(u,v_0)=1$. For a given $c\in \Fq^*$, the set of vectors $v$ such that $f_L(u,v)=c$ are exactly $cv_0+bu$ for $b\in \Fq$,

Thus we have $T(u,cv_0+bu)=\{(1+b)u+cv_0,(\omega+b\omega^2)u+\omega^2cv_0,(\omega^2+b\omega)u+\omega cv_0\}$, thus $\bigcup_{f_L(u,v)=c}T(u,v)=\bigcup_{b\in\Fq}T(u,cv_0+bu)=\{bu+\omega^acv_0|a=0,1,2,b\in\Fq\}$

Therefore $\bigcup_{f_L(u,v)=g^mc}(\pm T(u,v))=\{\pm bu\pm\omega^ag^mcv_0|a=0,1,2,b\in\Fq,0\leq m<(q-1)/6\}=\{bu+av_0|b\in\Fq,a\in\Fq^{*}\}=L \setminus K$.
\end{proof}

\begin{lemma}Let $L_{K}$ be the set of all 2-dimensional subspaces of $W$ containing $K$ (so that $|L_K|=\frac{q^{n-1}-1}{q-1}$), then the sets $\{L\setminus K|{L\in L_{K}}\}$ give a partition of $W\setminus K$.
\end{lemma}
\begin{cor}\label{lem:partialclass2}
The set $P_{u,c}=\bigcup_{L\in U_{K_i}}P_{u,L,c}$ gives a partition of $W\setminus K_i$.
\end{cor}

We will now construct the frame from the partial parallel classes defined above.

Let $U=\{K_i|1\leq i\leq (q^n-1)/(q-1)\}$ be the set of 1-dimensional subspaces of $W$. We choose a generator $u_i\in K_i$ in each $K_i$, and denote $V=\{u_i|1\leq i\leq (q^n-1)/(q-1)\}$.

\begin{lemma}\label{lem:blocks}
The union of partial parallel classes $P_{g^au_i,\omega^b}$ where $0\leq a<t$, $b=0,1,2$ and $1\leq i\leq q+1$ contains every zero-sum non-collinear triple exactly once.
\end{lemma}
\begin{proof}
The total number of triples in these partial parallel classes is $\frac{q(q-1)}{3}\cdot\frac{q^{n-1}-1}{q-1}\cdot\frac{q-1}{6}\cdot 3\cdot\frac{q^n-1}{q-1}=(q^n-q)(q^n-1)/6$, which agrees with the total number of zero-sum non-collinear triples. We only need to prove that every such triple $\{x,y,-x-y\}$ belongs to at least one of these classes.

Let $u_0=(1-\omega^2)(\omega x-y)/3$ and $v_0=(1-\omega^2)(\omega y-x)/3$ so that we have $\{x,y,-x-y\}=T(u_0,v_0)$. Let $L$ be the subspace spanned by $x$ and $y$, we have $f_L(u_0,v_0)=(\omega^2-\omega)f_L(x,y)/3$. Since $f_L(u_0,v_0)=-f_L(v_0,u_0)$, without loss of generality, we can assume that $f_L(u_0,v_0)=g^m\omega^b$ with $0\leq m<(q-1)/6$.

Assume now that $u_0=g^du_i$ with $u_i\in V$ and $0\leq d<q-1=6t$. We write $d=ct+a$ where $c=0,1,2,3,4,5$ and $0\leq a<t$. We can now conclude that $\{x,y,-x-y\}=T(u_0,v_0)=(-1)^c T(\omega^{-c}g^au_i,v_0)=(-1)^c T(g^au_i,\omega^{-c}v_0)$ belongs to $P_{g^au_i,L,\omega^b}$ and therefore $P_{g^au_i,\omega^b}$.

\end{proof}

It is an immediate conclusion from the above lemma that the union of all translations of the partial parallel classes contains every non-collinear triple in $W$.

We now proceed to fill the holes in the frame. For each $K_i\in U$, we construct a system of representatives $R_i$ of the cosets of $K_i$, and an associated function $p_i:W\rightarrow\Fq$ such that $x-p_i(x)u_i\in R_i$ for all $x\in W$.

Recall that $Q_{i,j}$ are the parallel classes of $\mathcal{D}_i$ and that $\{\infty_1,\infty_2,i\}\in Q_{i,0}$.

\begin{theorem}\label{thm:KTS}
For each $w\in W$, the following parallel classes
\begin{equation}\label{eqn:classes}
P_{w,u_i,a+bt+1}=(w+P_{g^au_i,\omega^b})\bigcup\{w-p_i(w)u_i+Au_i|A\in Q_{p_i(w),a+bt+1}\}
\end{equation}
for $u_i\in V$, $0\leq a<t$ and $b=0,1,2$
and
\begin{equation}\label{eqn:lastclass}
P_{w,*}=\bigcup_{u_i\in V}\{w-p_i(w)u_i+Au_i|A\in Q_{p_i(w),0}\}
\end{equation}
constitutes a KTS($q^n+2$), which we will denote by $\mathcal{B}_w$.
\end{theorem}
\begin{proof}
By Lemma \ref{lem:partialclass}, $w+P_{g^au_i,\omega^b}$ is a partition of $W\setminus (w+K_i)$, while $\{w-p_i(w)x_i+Au_i|A\in Q_{p_i(w),a+bt+1}\}$ is a partition of $\{\infty_1,\infty_2\}\bigcup(w+K_i)$. Therefore $P_{w,u_i,a+bt+1}$ is indeed a parallel class. On the other hand, for each $u_i\in V$, $\{w-p_i(w)x_i+Au_i|A\in Q_{p_i(w),0}\}$ contains the triple $\{w,\infty_1,\infty_2\}$ as well as other $(q-1)/3$ triples that gives a partition of $w+(K_i\setminus{0})$, so the union of these is also a parallel class.

We now prove that $\mathcal{B}_w$ is an STS, and therefore a KTS by the argument above. The number of parallel classes in $\mathcal{B}_w$ equals $(q^n+1)/2$, which agrees with the expected number of a KTS($q^n+2$), so it suffices to find a triple $B$ containing $\{x,y\}$ in one of these classes for all $x,y\in X$.

We distinguish the following cases.

\begin{enumerate}
  \item $\{x,y\}=\{\infty_1,\infty_2\}$. We know that $\{w,\infty_1,\infty_2\}\in P_{w,*}$, so we can take $B=\{w,\infty_1,\infty_2\}$.
  \item $y\in \{\infty_1,\infty_2\}$, $x\in W$. If $x=w$ we take $B=\{w,\infty_1,\infty_2\}$ as above. Otherwise, let $x=au_i+w$ where $a\in\mathbb{F}_q$, then there is a unique block $B_0$ containing $\{a-p_i(w),y\}$ in $\mathcal{D}_{p_i(w)}$. Take $B=w-p_i(w)u_i+B_0u_i$.
  \item $\{x,y\}\subset W$, and $x-w,y-w$ are linear dependent. Similarly, let $x=au_i+w, y=bu_i+w$ where $a,b\in\mathbb{F}_q$, there is a unique block $B_0$ containing $\{a-p_i(w),b-p_i(w)\}$ in $\mathcal{D}_{p_i(w)}$. Take $B=w-p_i(w)u_i+B_0u_i$.
  \item $\{x,y\}\subset W$, and $x-w,y-w$ are linear independent. By Lemma \ref{lem:blocks}, the triple $\{x-w,y-w,2w-x-y\}$ belongs to one of the partial parallel classes $P_{g^au_i,\omega^b}$. Take $B=\{x,y,3w-x-y\}\in w+P_{g^au_i,\omega^b} \subset P_{w,u_i,a+bt+1}$.
\end{enumerate}
\end{proof}

\begin{theorem}\label{thm:LKTS}
The set $\{\mathcal{B}_w\}_{w\in W}$ is an LKTS($q^n+2$).
\end{theorem}
\begin{proof}
We only need to find, for each triple $\{x,y,z\}\subset X$, an element $w\in W$ such that $\{x,y,z\}\in \mathcal{B}_w$.

Like the proof of the previous lemma, we distinguish four cases.
\begin{enumerate}
  \item $\{y,z\}=\{\infty_1,\infty_2\}$. Take $w=x$.
  \item $\{x,y\}\subset W$, $z\in \{\infty_1,\infty_2\}$. Let $x=r+ax_i,y=r+bx_i$ where $a,b\in\mathbb{F}_q$ and $r\in R_i$, then there is a unique $d\in\mathbb{F}_q$ such that $\{a,b,z\}\in\mathcal{D}_d$. Take $w=r+dx_i$.
  \item $\{x,y,z\}\subset W$, and are collinear. Similarly, let $x=r+ax_i,y=r+bx_i,z=r+cx_i$ where $a,b,c\in\mathbb{F}_q$ and $r\in R_i$, there is a unique $d\in\mathbb{F}_q$ such that $\{a,b,c\}\in\mathcal{D}_d$. Take $w=r+dx_i$.
  \item $\{x,y,z\}\subset W$, and are not collinear. Take $w=(x+y+z)/3$.
\end{enumerate}
\end{proof}

\begin{rem}\label{rem:trans} Suppose the original large set $\{\mathcal{D}_i\}_{i\in\Fq}$ is invariant under translations in $\Fq$, i.e. $\mathcal{D}_i=\mathcal{D}_0+i$ and $Q_{i,j}=Q_{0,j}+i$ holds for all $i$ and $j$. For all $w\in W$ and $A\in Q_{p_i(w),j}$, we define $A_0=A-p_i(w)+p_i(0)$ so that $A_0\in Q_{p_i(0),j}$, and therefore $w-p_i(w)u_i+Au_i=w-p_i(0)+A_0u_i$. By (\ref{eqn:classes}) and (\ref{eqn:lastclass}), we can infer that $\mathcal{B}_w=\mathcal{B}_0+w$, thus the large set $\{\mathcal{B}_w\}_{w\in W}$ is also invariant under translations in $W$. We can also prove that different choice of $R_i$ gives the same design in this case.
\end{rem}

\section{Example: construction of an LTKS(171)}\label{sec:eg}

In this section, we present an example to illustrate our construction, applied to the case $q=13$ and $n=2$. The resulting large set is of order $13^2+2=171$, and is previously unknown. The basis of this construction is the LKTS(15) by Denniston \cite{MR0369086}, as described below:

\begin{eqnarray*}
Q_{0,0}&=&\{\{\infty_1,\infty_2,0\},\{1,4,5\},\{2,6,11\},\{3,7,10\},\{8,9,12\}\} \\
Q_{0,1}&=&\{\{\infty_1,1,6\},\{\infty_2,2,8\},\{0,10,12\},\{3,5,9\},\{4,7,11\}\} \\
Q_{0,2}&=&\{\{\infty_1,2,5\},\{\infty_2,4,9\},\{0,3,11\},\{1,7,12\},\{6,8,10\}\} \\
Q_{0,3}&=&\{\{\infty_1,3,12\},\{\infty_2,5,7\},\{0,4,6\},\{1,8,11\},\{2,9,10\}\} \\
Q_{0,4}&=&\{\{\infty_1,4,10\},\{\infty_2,11,12\},\{0,5,8\},\{1,2,3\},\{6,7,9\}\} \\
Q_{0,5}&=&\{\{\infty_1,7,8\},\{\infty_2,3,6\},\{0,1,9\},\{2,4,12\},\{5,10,11\}\} \\
Q_{0,6}&=&\{\{\infty_1,9,11\},\{\infty_2,1,10\},\{0,2,7\},\{3,4,8\},\{5,6,12\}\}
\end{eqnarray*}

and $Q_{i,j}=Q_{0,j}+i$ for $i\in\mathbb{F}_{13}$. We denote the point $(a,b)\in\mathbb{F}_{13}^2$ by $a_b$, and choose $g=2$, $\omega=g^4=3$, $V=\{1_0,0_1,1_1,2_1,3_1,4_1,5_1,6_1,7_1,8_1,9_1,10_1,11_1,12_1\}$, and the skew-symmetric bilinear form $f(a_b,c_d)=ad-bc$. By Remark \ref{rem:trans}, the 169 designs can be obtained as translations of a base design $\mathcal{B}_{0_0}$ on the set $X=\{\infty_1,\infty_2\}\cup\mathbb{F}_{13}^2$.

By (\ref{eqn:lastclass}), we obtain the first parallel class of $\mathcal{B}_{0_0}$:

\begin{align*}
P_{0_0,*}= \{&\{\infty_1, \infty_2, {0}_{0}\}, \{1_{0},4_{0},5_{0}\}, \{3_{0},7_{0},10_{0}\}, \{2_{0},6_{0},11_{0}\}, \{8_{0},9_{0},12_{0}\}, \\
&\{0_{1},0_{4},0_{5}\}, \{0_{3},0_{7},0_{10}\}, \{0_{2},0_{6},0_{11}\}, \{0_{8},0_{9},0_{12}\}, \\
&\{1_{1},4_{4},5_{5}\}, \{3_{3},7_{7},10_{10}\}, \{2_{2},6_{6},11_{11}\}, \{8_{8},9_{9},12_{12}\}, \\
&\{2_{1},8_{4},10_{5}\}, \{6_{3},1_{7},7_{10}\}, \{4_{2},12_{6},9_{11}\}, \{3_{8},5_{9},11_{12}\}, \\
&\{3_{1},12_{4},2_{5}\}, \{9_{3},8_{7},4_{10}\}, \{6_{2},5_{6},7_{11}\}, \{11_{8},1_{9},10_{12}\}, \\
&\{4_{1},3_{4},7_{5}\}, \{12_{3},2_{7},1_{10}\}, \{8_{2},11_{6},5_{11}\}, \{6_{8},10_{9},9_{12}\}, \\
&\{5_{1},7_{4},12_{5}\}, \{2_{3},9_{7},11_{10}\}, \{10_{2},4_{6},3_{11}\}, \{1_{8},6_{9},8_{12}\}, \\
&\{6_{1},11_{4},4_{5}\}, \{5_{3},3_{7},8_{10}\}, \{12_{2},10_{6},1_{11}\}, \{9_{8},2_{9},7_{12}\}, \\
&\{7_{1},2_{4},9_{5}\}, \{8_{3},10_{7},5_{10}\}, \{1_{2},3_{6},12_{11}\}, \{4_{8},11_{9},6_{12}\}, \\
&\{8_{1},6_{4},1_{5}\}, \{11_{3},4_{7},2_{10}\}, \{3_{2},9_{6},10_{11}\}, \{12_{8},7_{9},5_{12}\}, \\
&\{9_{1},10_{4},6_{5}\}, \{1_{3},11_{7},12_{10}\}, \{5_{2},2_{6},8_{11}\}, \{7_{8},3_{9},4_{12}\}, \\
&\{10_{1},1_{4},11_{5}\}, \{4_{3},5_{7},9_{10}\}, \{7_{2},8_{6},6_{11}\}, \{2_{8},12_{9},3_{12}\}, \\
&\{11_{1},5_{4},3_{5}\}, \{7_{3},12_{7},6_{10}\}, \{9_{2},1_{6},4_{11}\}, \{10_{8},8_{9},2_{12}\}, \\
&\{12_{1},9_{4},8_{5}\}, \{10_{3},6_{7},3_{10}\}, \{11_{2},7_{6},2_{11}\}, \{5_{8},4_{9},1_{12}\}\} \\
\end{align*}

We describe one of the other 84 classes, namely $P_{0_0,u_1=1_0,1}$. By (\ref{eqn:classes}), this class is given by
$P_{0_0,u_1,1}=P_{u_1,1}\bigcup\{-p_1(0_0)u_1+Au_1|A\in Q_{p_1(0_0),1}\}=P_{u_1,1}\bigcup\{Au_1|A\in Q_{0,1}\}$. The complete description of the class is given below:

\begin{align*}P_{0_0,u_{1},1}= \{&\{\infty_1,1_{0},6_{0}\}, \{\infty_2,2_{0},8_{0}\}, \{0_{0},10_{0},12_{0}\}, \{3_{0},5_{0},9_{0}\}, \{4_{0},7_{0},11_{0}\}, \\
&\{1_{1},3_{9},9_{3}\}, \{12_{12},10_{4},4_{10}\}, \{1_{2},3_{5},9_{6}\}, \{12_{11},10_{8},4_{7}\}, \\
&\{2_{1},12_{9},12_{3}\}, \{11_{12},1_{4},1_{10}\}, \{2_{2},12_{5},12_{6}\}, \{11_{11},1_{8},1_{7}\}, \\
&\{3_{1},8_{9},2_{3}\}, \{10_{12},5_{4},11_{10}\}, \{3_{2},8_{5},2_{6}\}, \{10_{11},5_{8},11_{7}\}, \\
&\{4_{1},4_{9},5_{3}\}, \{9_{12},9_{4},8_{10}\}, \{4_{2},4_{5},5_{6}\}, \{9_{11},9_{8},8_{7}\}, \\
&\{5_{1},0_{9},8_{3}\}, \{8_{12},0_{4},5_{10}\}, \{5_{2},0_{5},8_{6}\}, \{8_{11},0_{8},5_{7}\}, \\
&\{6_{1},9_{9},11_{3}\}, \{7_{12},4_{4},2_{10}\}, \{6_{2},9_{5},11_{6}\}, \{7_{11},4_{8},2_{7}\}, \\
&\{7_{1},5_{9},1_{3}\}, \{6_{12},8_{4},12_{10}\}, \{7_{2},5_{5},1_{6}\}, \{6_{11},8_{8},12_{7}\}, \\
&\{8_{1},1_{9},4_{3}\}, \{5_{12},12_{4},9_{10}\}, \{8_{2},1_{5},4_{6}\}, \{5_{11},12_{8},9_{7}\}, \\
&\{9_{1},10_{9},7_{3}\}, \{4_{12},3_{4},6_{10}\}, \{9_{2},10_{5},7_{6}\}, \{4_{11},3_{8},6_{7}\}, \\
&\{10_{1},6_{9},10_{3}\}, \{3_{12},7_{4},3_{10}\}, \{10_{2},6_{5},10_{6}\}, \{3_{11},7_{8},3_{7}\}, \\
&\{11_{1},2_{9},0_{3}\}, \{2_{12},11_{4},0_{10}\}, \{11_{2},2_{5},0_{6}\}, \{2_{11},11_{8},0_{7}\}, \\
&\{12_{1},11_{9},3_{3}\}, \{1_{12},2_{4},10_{10}\}, \{12_{2},11_{5},3_{6}\}, \{1_{11},2_{8},10_{7}\}, \\
&\{0_{1},7_{9},6_{3}\}, \{0_{12},6_{4},7_{10}\}, \{0_{2},7_{5},6_{6}\}, \{0_{11},6_{8},7_{7}\}\} \\
\end{align*}

For a complete list of the 85 parallel classes of $\mathcal{B}_{0_0}$, see the appendix.

\section*{References}
\bibliographystyle{siam}
\bibliography{designs}

\section*{Appendix}
Here we give a complete list of the 85 parallel classes of the base design $\mathcal{B}_{0_0}$ constructed in Section \ref{sec:eg}. The point $a_b$ is represented as \verb"ab", A,B,C denotes 10, 11 and 12 respectively, and $\infty_i$ are denoted by \verb"XX" and \verb"YY". 

\include{appendix1}

\end{document}